\documentclass{techrep}

\usepackage[scr=rsfso]{mathalfa}

\usepackage{parskip}
\usepackage{upgreek}

\usepackage{mathtools}
\newcommand*{\myov}[1]{\overbracket[0.65pt][-1pt]{#1}}

\renewcommand{\vec}[1]{{\mathchoice
                     {\mbox{\boldmath$\displaystyle{#1}$}}
                     {\mbox{\boldmath$\textstyle{#1}$}}
                     {\mbox{\boldmath$\scriptstyle{#1}$}}
                     {\mbox{\boldmath$\scriptscriptstyle{#1}$}}}}

\newcommand{\mat}[1]{\mathsf{\mathbf{#1}}}
\newcommand{\GM}[2]{\mathcal{N}\!\left( {#1}, {#2}\right)}
\newcommand{\bias}{\mathrm{Bias}}
\newcommand{\RISK}{\mathrm{Risk}}
\newcommand{\BayesRISK}{\myov{\mathrm{Risk}}}
\newcommand{\MSE}{\mathrm{MSE}}
\newcommand{\obs}{\vec{y}} 
\newcommand{\prcov}{\vec{\Gamma}_{\mathrm{pr}}}
\newcommand{\postcov}{\vec{\Gamma}_{\mathrm{post}}}
\newcommand{\noise}{\vec{\Gamma}_{\mathrm{noise}}}
\newcommand{\reg}{\mat{R}}
\newcommand{\ave}[2]{\mathsf{E}_{{#1}}\left\{ {#2} \right\}}

\newcommand{\norm}[1]{\left\| {#1} \right\|}
\newcommand{\trace}{\mathsf{tr}}

\newcommand{\like}{\pi_{\scriptscriptstyle\text{like}}}

\newcommand{\R}{\mathbb{R}}

\newcommand{\F}{\mat{F}}
\newcommand{\HN}{\mat{H}}
\newcommand{\HM}{\mat{H}_{\mathrm{misfit}}}

\newcommand{\B}{\mathcal{B}}
\renewcommand{\L}[1]{\mathrm{Loss}\left[ {#1} \right]}
\newcommand{\dparh}{\widehat{\dpar}}
\newcommand{\dparpr}{\dpar_{\mathrm{pr}}}
\newcommand{\dparpost}{\dpar_{\mathrm{post}}}
\newcommand{\priorm}{\mu_\text{pr}}
\renewcommand{\Xi}{X^{-1}}
\newcommand{\dpar}{\vec{m}}
\newcommand{\aversmall}[1]{\mathsf{E}\{ {#1} \}}
\newcommand{\aver}[1]{\mathsf{E}\left\{ {#1} \right\}}
\newcommand{\avey}[1]{\mathsf{E}_{{\vec{y}|\dpar}}\left\{ {#1} \right\}}
\newcommand{\regpar}{\upbeta}
\newcommand{\tran}{{\mkern-1.5mu\mathsf{T}}}                
\newcommand{\Law}[1]{\mathscr{L}_{{#1}}}
\newcommand{\ip}[2]{\langle {#1}, {#2}\rangle}
\newcommand{\ipn}[3]{\langle {#1}, {#2}\rangle_{{\scriptscriptstyle{#3}}}}

\newcommand{\hilb}{\mathscr{H}}
\newcommand{\Y}{\mathscr{Y}}

\newcommand{\cc}{\vec{c}}
\newcommand{\hh}{\vec{h}}
\renewcommand{\ss}{\vec{s}}
\newcommand{\xx}{\vec{x}}

\newcommand{\zz}{\vec{z}}
\renewcommand{\AA}{\mat{A}}
\newcommand{\QQ}{\mat{Q}}

\TITLE{On Bayes risk of the posterior mean in linear inverse problems}
\SHORTTITLE{Bayes risk in linear inverse problem}
\AUTHORS{Alen~Alexanderian\footnote{North Carolina State University, Raleigh, NC, USA.
         \THANKS{alexanderian@ncsu.edu}{\today}
         This work was supported in part by US National Science Foundation grant \#2111044.}}
\ABSTRACT{We consider the concept of Bayes risk in the
context of finite-dimensional ill-posed linear inverse problem with Gaussian
prior and noise models. In this note, we rederive the following 
well-known result: 
in the present Gaussian linear setting, the Bayes risk of the posterior mean, relative
to the sum of squares loss function, equals the trace of the posterior
covariance operator.
While our discussion is limited to a finite-dimensional setup, our presentation
is motivated by more general infinite-dimensional formulations. 
}

\begin{document}
\maketitle

\section{Introduction}
In this note, we consider some basics from theory of point estimation regarding
the notion of the risk of an estimator. Then, taking a Bayesian point of view,
we discuss the notion of Bayes risk.  We focus on finite-dimensional (discretized) 
ill-posed linear inverse problems with Gaussian prior and additive Gaussian noise models.
In this case, we show that the Bayes risk of the posterior mean, relative to the
sum of squares loss function, equals the trace of the posterior covariance
operator; see Proposition~\ref{prp:main}.  Note that in the present linear
Gaussian setting, the posterior mean coincides with the maximum a posteriori
probability (MAP) estimator of the inversion parameter vector.

In a Bayesian inverse problem,
the notion of Bayes risk can be used as a means of assessing the statistical quality of the
estimated parameter. This is important, for example, in the context of optimal
experimental design (OED). In an OED problem, one seeks measurements that optimize
the quality of the estimated parameters.  Bayes risk is one from among several
possibilities for an optimal experimental design criterion; see
e.g,~\cite{ChalonerVerdinelli95,Haber08,Alexanderian21}. 

The results discussed herein are well-known. The purpose of this note is to
provide an accessible and self-contained summary. For further details on the
concepts of risk and Bayes risk of an estimator, see
e.g.,~\cite{Ghosh06,Berger13}.  We consider only inverse problems with
finite-dimensional parameters.  However, the result on Bayes risk
(Proposition~\ref{prp:main}) can be extended to Bayesian linear inverse problems
in infinite-dimensional Hilbert spaces~\cite{AlexanderianGloorGhattas16}.  In
fact, our style of presentation, while limited to a finite-dimensional setup, is
motivated by the more general infinite-dimensional Hilbert space setting. 

\section{Preliminaries}

In this section, we recall some basics from probability and statistics that are
needed in this note.  Throughout, $\R^N$ is always equipped with the
Euclidean inner product, denoted by $\ip{\cdot}{\cdot}$, and the Euclidean norm
$\| \cdot \| = \ip{\cdot}{\cdot}^{1/2}$. If we want to emphasize the dimension
$N$ of the space, when using the Euclidean inner product, we use the notation
$\ipn{\cdot}{\cdot}{N}$. 

\subsection{Background concepts from probability}
We denote the Borel sigma-algebra on $\R^N$ by $\B(\R^N)$.
Let us recall the definition of the mean and covariance operator of a
probability measure $\mu$ on $(\R^N, \B(\R^N))$ that has finite first and 
second moments.
Let $\{\vec{e}_1, \ldots, \vec{e}_N\}$ be the
standard basis in $\R^N$.  The mean $\cc$ and covariance operator
$\QQ$ of $\mu$ are defined as follows:
\[
c_i = \ip{\vec{e}_i}{\cc} = \int_{\R^N} \ip{\vec{e}_i}{\vec{x}} \, 
\mu(d\vec{x}) \quad \mbox{ and }\quad
Q_{ij} = \int_{\R^N} \ip{\vec{e}_i}{\vec{x}-\vec{c}} \ip{\vec{e}_j}{\vec{x}-\vec{c}} \, \mu(d\vec{x}),
\] 
with $i, j \in \{1, \ldots, N\}$.\footnote{The notions of mean and covariance operator
of measures 
extend naturally to the case of probability measures on infinite-dimensional
Hilbert spaces~\cite{DaPrato}.}  Using the definition of mean 
and covariance operator, it is straightforward to show 
\[
    \int_{\R^N} \| \xx \|^2 \,d\xx = \trace(\QQ) + \| \cc\|^2.
\]

Let $\mu$ be a probability measure on $(\R^N, \B(\R^N))$.  The law of a random
variable $\vec Z:(\R^N, \B(\R^N), \mu) \to (\R^M, \B(\R^M))$ is the probability
measure $\Law{\vec Z}$ defined by 
\[
    \Law{\vec Z}(E) = \mu(\vec{Z}^{-1}(E))
\quad E \in \B(\R^M).
\]
See~\cite{Williams} for more details.  We recall the following well-known 
result:
a Borel measurable function $\phi:\R^M \to \R$ 
is integrable with respect to $\Law{\vec{Z}}$
precisely when $\phi \circ \vec{Z}$ is integrable with respect to 
$\mu$. In addition, we have the change of variable formula
\[
\int_{\R^N} \phi(\vec{Z}(\xx)) \mu(d\xx) = \int_{\R^M} \phi(\zz) \Law{\vec{Z}}(d\zz).
\]
This result holds in more general measure theoretic settings;
see~\cite[Section 3.6]{Bogachev07}.
We will also need the following well-known result regarding 
Gaussian measures~\cite{DaPrato}.
We include a proof of this in the
present finite-dimensional setting, for completeness. 
\begin{proposition}\label{prp:change_meas_gaussian}
Let $\mu = \GM{\cc}{\QQ}$ be a Gaussian measure on $(\R^N, \B(\R^N))$. 
Let $\AA \in \R^{M \times N}$ and $\ss \in \R^M$.
Define the random variable $\vec{Z}(\vec x) = \AA \xx + \ss$ 
on $(\R^N, \B(\R^N), \mu)$.
Then, $\vec{Z}$ is a Gaussian random variable with law $\Law{\vec{Z}} = 
\GM{\AA\cc + \ss}{\AA \QQ \AA^\tran}$.
\end{proposition}
\begin{proof}
See the Appendix~\ref{appdx:proof}.
\end{proof}
Let us also recall the following related result: 
let $\vec{X}$ be a $N$-dimensional 
Gaussian random vector with law
$\GM{\cc}{\QQ}$ and let $\AA \in \R^{M \times N}$ and $\ss \in \R^M$. Then,
$\vec{Z} = \AA \vec{X} + \ss$ is also a Gaussian and 
\begin{equation}\label{equ:lin_image}
   \vec{Z} \sim \GM{\AA\cc + \ss}{\AA \QQ \AA^\tran}. 
\end{equation}
See e.g.,~\cite[p.~121]{Gut09} for a direct proof of this for 
Gaussian random vectors.

\subsection{Frequentist risk}
In what follows, $\hilb$ denotes the $N$-dimensional Euclidean space, 
$\dpar \in \hilb$ is an inference parameter, and $\dparh(\vec{y})$ is
an estimator (point-estimate) for $\dpar$ obtained using a realization of data $\vec{y}$. We 
let $\L{\dpar, \dparh(\vec{y})}$ denote a loss function.
An example (and that's the only one we consider) is the sum of 
squares loss function:
\[
   \L{\dpar, \dparh(\vec{y})} = \norm{ \dpar - \dparh(\vec{y})}^2,
\]
where $\| \cdot \|$ is the Euclidean norm.
The data $\obs$ takes values in $\Y = \R^D$. Herein, 
we let $\like(\vec{y} | \dpar)$ be a likelihood function.

From a frequentist point of view, there is nothing random about the parameter
vector $\dpar$. This parameter is an underlying truth we seek to estimate.
On the other hand, the estimator $\dparh$ is a random
variable---it is a function of the random variable $\vec{y}$.  (The 
distribution of $\vec{y}$ is specified by a likelihood model.)  We denote the
(frequentist) \emph{risk} of the estimator
$\dparh$ by $\RISK(\dparh; \dpar)$ and define it as follows,
\[
   \RISK(\dparh; \dpar) = \ave{\vec{y}|\dpar}{\L{\dpar, \dparh(\cdot)}} = \int_\Y \L{\dpar, \dparh(\vec{y})} \like( \vec{y} | \dpar) \, d\vec{y}. 
\]
A common choice for the loss function is the sum of squares loss function, where
we use
\[
   \RISK(\dparh; \dpar) = 
   \ave{\vec{y}|\dpar} {\norm{\dpar - \dparh(\vec{y})}^2} 
   = \int_\Y \norm{\dpar - \dparh(\vec{y})}^2  \like( \vec{y} | \dpar) \, d\vec{y}.
\]
The above risk is often referred to as the mean-squared-error
and is denoted by $\MSE$; i.e., $\MSE(\dparh; \dpar) = \avey{\norm{\dpar - \dparh(\vec{y})}^2}$.
An elementary calculation shows, 
\begin{equation}\label{equ:MSE}
    \MSE(\dparh; \dpar) = \norm{\dpar - \aver{\dparh}}^2 + \aver{ \norm{\dparh - \aver{\dparh}}^2},
\end{equation}
where for brevity we have used $\aver{\cdot}$ in place of $\avey{\cdot}$.
Note that the first term in~\eqref{equ:MSE} quantifies the magnitude of estimation bias, 
$\norm{\dpar - \aver{\dparh}}^2 = \norm{\bias(\dparh)}^2$, 
and the second one describes the variability of the estimator around its mean. 

\section{The $\MSE$ for the solution of an ill-posed linear inverse problem}
Consider a linear forward operator $\F \in \R^{D \times N}$ and 
the following data model 
\[
    \obs = \F \dpar + \vec{\eta},
\]
where $\vec{\eta}$ is a centered Gaussian with covariance $\noise$. This implies, 
$\obs | \dpar \sim \GM{\F\dpar}{\noise}$.
Given a realization of $\obs$, we can 
obtain a point-estimate for the parameter $\dpar$ by computing 
\begin{equation}\label{equ:linearinv}
    \dparh(\obs) = \arg\min_{\dpar \in \hilb}\, 
\frac12 (\F \dpar - \obs)^\tran \noise^{-1} (\F\dpar - \obs) + 
\frac\regpar2 (\dpar - \dpar_0)^\tran \reg (\dpar - \dpar_0).
\end{equation}
Here $\reg \in \R^{N \times N}$ is a (symmetric positive definite) regularization operator, needed 
due to ill-posed nature of the inverse problemm, $\regpar$ is a regularization 
parameter, and $\dpar_0$ is a reference parameter value. 
In many cases $\dpar_0 = \vec{0}$ is used.
The solution to the regularized least-squares problem 
is 
\begin{equation}\label{equ:dparh}
    \dparh(\obs) = \HN^{-1}(\F^\tran\noise^{-1}\obs + \regpar \reg \dpar_0),
\end{equation}
with 
\[
    \HN = \HM + \regpar \reg, \qquad \mbox{ where }\qquad \HM = \F^\tran \noise^{-1} \F.
\]
Note that $\HN$ is the Hessian of the objective function in~\eqref{equ:linearinv}, 
and $\HM$ is the Hessian of the data-misfit part in the objective function.

The following result describes the $\MSE(\dparh; \dpar)$ in the
present setting: 
\begin{proposition}
Let $\dparh$ in~\eqref{equ:dparh} be the estimator of $\dpar$. Then,
\begin{equation}\label{equ:MSE_linear_inv}
    \MSE(\dparh; \dpar) = \regpar^2\norm{\HN^{-1}\reg (\dpar -\dpar_0)}^2 
                             + \trace(\HN^{-2} \HM).
\end{equation}
\end{proposition}
\begin{proof}
First we note that
\[
\begin{aligned}
   \bias(\dparh) = \avey{\dparh} - \dpar 
   &= \avey{\HN^{-1}\F^\tran\noise^{-1}\obs} + \regpar \HN^{-1} \reg \dpar_0 - \dpar \\
   &= \HN^{-1}\F^\tran\noise^{-1}\F\dpar - \dpar + \regpar \HN^{-1} \reg \dpar_0\\
   &= \HN^{-1}( \F^\tran \noise^{-1} \F - \HN) \dpar + \regpar \HN^{-1} \reg \dpar_0\\ 
   &= -\regpar\HN^{-1}\reg(\dpar - \dpar_0).
\end{aligned}
\]
Thus, the first term in~\eqref{equ:MSE} is given by $\regpar^2
\norm{\HN^{-1}\reg (\dpar-\dpar_0)}^2$.
Next, we compute the second term in~\eqref{equ:MSE}. To this end, let us
define $\vec{Z}(\obs) = \dparh(\obs) - \aver{\dparh(\cdot)}$, and note that $\vec{Z}$ 
has a 
Gaussian law $\Law{\vec Z}$ with mean
zero and covariance operator,
\[
    (\HN^{-1} \F^\tran \noise^{-1}) \noise (\HN^{-1} \F^\tran \noise^{-1})^\tran = \HN^{-1} \F^\tran \noise^{-1} \F \HN^{-1} 
      = \HN^{-1} \HM \HN^{-1}.
\]
Therefore, 
\[
    \aversmall{\norm{\vec{Z}(\obs)}^2} = \int_\hilb \norm{\vec{z}}^2 
\, \Law{z}(d\vec z)
    = \trace(\HN^{-1} \HM \HN^{-1}) 
    = \trace(\HN^{-2} \HM). \qedhere
\]
\end{proof}

\section{Bayes risk and the case of the Bayesian linear inverse problem}
As seen in the previous section, the $\MSE$ of the least-squares estimate
$\dparh$ has a contribution from the estimation bias; see, 
e.g.,~\eqref{equ:MSE_linear_inv}, where the first term in the sum 
is the size of the estimation bias.  We cannot compute the
bias term because it depends on the unknown parameter $\dpar$ we are trying
estimate. 
As seen shortly, this issue can be dealt with naturally in a Bayesian
framework, where we view $\dpar$ as a random variable.

We equip $\hilb$ with the corresponding Borel sigma algebra and consider the
measurable space $(\hilb, \B(\hilb))$.  In the Bayesian framework, our prior
knowledge regarding the parameter $\dpar$ is encoded in a probability measure
$\priorm$ on $(\hilb, \B(\hilb))$. We refer to $\priorm$ as the prior measure.
This setup, in particular, enables computing the expected value of the bias
term and leads to the notion of \emph{Bayes risk}.  Namely, 
we define the Bayes risk as the expected value of the
frequentist risk with respect to the prior law of $\dpar$.  
In particular, using the $\MSE$ as the risk measure, we define the Bayes risk 
of the estimator $\dparh$ by
\[
   \BayesRISK(\dparh) = \ave{\priorm}{\MSE(\dparh; \dpar)}
                 = \int_\hilb \int_\Y \norm{\dpar - \dparh(\vec{y})}^2  \like( \vec{y} | \dpar) \, d\vec{y} \, \priorm(d\dpar).
\]
We consider a Gaussian prior measure $\priorm = \GM{\dparpr}{\prcov}$.

Next, we cast the linear inverse problem~\eqref{equ:linearinv} in a Bayesian
framework.  The following is based on standard results on Gaussian linear inverse
problems~\cite{Tarantola05}. 
In the Bayesian formulation, we relabel components of
~\eqref{equ:linearinv} and~\eqref{equ:dparh} as follows:
\begin{itemize}
\item $\dpar_0$ is replaced with $\dparpr$;
\item $\regpar\reg$ is replaced with $\prcov^{-1}$;
\item $\HN^{-1}$ is replaced with the covariance operator $\postcov$; and
\item $\dparh(\obs)$ becomes the posterior mean $\dparpost^\obs$.
\end{itemize}
The expression~\eqref{equ:MSE_linear_inv} for the $\MSE$, cast in the present Bayesian framework, is given by
\[
    \MSE(\dparpost^\obs; \dpar) = \norm{\postcov\prcov^{-1} (\dpar-\dparpr)}^2 + \trace(\postcov^{2} \HM).    
\]
The following result, which is the main point of this note, concludes our discussion.
\begin{proposition}\label{prp:main}
We have $\BayesRISK(\dparpost^\obs) = \trace(\postcov)$.
\end{proposition}
\begin{proof}
Consider the random variable $\vec Z(\dpar) = 
\postcov\prcov^{-1} (\dpar - \dparpr)$, defined on $(\hilb,\B(\hilb), \priorm)$, 
and note that $\Law{\vec Z} =  \GM{\vec 0}{\postcov\prcov^{-1}\postcov}$. Thus, 
\[
\int_\hilb \norm{\postcov\prcov^{-1} (\dpar - \dparpr)}^2 \, \priorm(d\dpar)
   = \int_\hilb \norm{ \vec{z}}^2 \, \Law{\vec{z}} (d\vec{z}) 
   = \trace(\postcov\prcov^{-1}\postcov)  
   = \trace(\postcov^{2}\prcov^{-1}).
\]
Therefore,
\[
\begin{aligned}
\BayesRISK(\dparpost^\obs) = \int_\hilb \MSE(\dparpost) \, \priorm(d\dpar)
                &= \int_\hilb \norm{\postcov\prcov^{-1} (\dpar - \dparpr)}^2 \, \priorm(d\dpar) + \trace(\postcov^{2} \HM) \\
                &= \trace(\postcov^{2}\prcov^{-1}) +  \trace(\postcov^{2} \HM) = \trace(\postcov).\qedhere
\end{aligned}
\]
\end{proof}

\bibliographystyle{plain}
\bibliography{refs}

\def\cprime{$'$}
\begin{thebibliography}{10}

\bibitem{Alexanderian21}
Alen Alexanderian.
\newblock Optimal experimental design for infinite-dimensional {B}ayesian
  inverse problems governed by {PDEs}: A review.
\newblock {\em Inverse Problems}, 37(4):043001, 2021.

\bibitem{AlexanderianGloorGhattas16}
Alen Alexanderian, Philip~J. Gloor, and Omar Ghattas.
\newblock On {B}ayesian {A}-and {D}-optimal experimental designs in infinite
  dimensions.
\newblock {\em Bayesian Analysis}, 11(3):671--695, 2016.

\bibitem{Berger13}
James~O Berger.
\newblock {\em Statistical decision theory and Bayesian analysis}.
\newblock Springer Science \& Business Media, 2013.

\bibitem{Bogachev07}
V.~I. Bogachev.
\newblock {\em Measure theory. {V}ol. {I}}.
\newblock Springer-Verlag, Berlin, 2007.

\bibitem{ChalonerVerdinelli95}
Kathryn Chaloner and Isabella Verdinelli.
\newblock Bayesian experimental design: A review.
\newblock {\em Statist. Sci.}, 10(3):273--304, 1995.

\bibitem{DaPrato}
Giuseppe Da~Prato.
\newblock {\em An introduction to infinite-dimensional analysis}.
\newblock Springer Science \& Business Media, 2006.

\bibitem{Ghosh06}
Jayanta~K Ghosh, Mohan Delampady, and Tapas Samanta.
\newblock {\em An introduction to Bayesian analysis: theory and methods},
  volume 725.
\newblock Springer, 2006.

\bibitem{Gut09}
Allan Gut.
\newblock {\em An intermediate course in probability}.
\newblock Springer Texts in Statistics. Springer, New York, second edition,
  2009.

\bibitem{Haber08}
Eldad Haber, Lior Horesh, and Luis Tenorio.
\newblock Numerical methods for experimental design of large-scale linear
  ill-posed inverse problems.
\newblock {\em Inverse Problems}, 24(5):055012, 2008.

\bibitem{Tarantola05}
Albert Tarantola.
\newblock {\em Inverse problem theory and methods for model parameter
  estimation}.
\newblock SIAM, 2005.

\bibitem{Williams}
David Williams.
\newblock {\em Probability with martingales}.
\newblock Cambridge Mathematical Textbooks. Cambridge University Press,
  Cambridge, 1991.

\end{thebibliography}

\clearpage
\appendix
\section{Proof of Propositon~\ref{prp:change_meas_gaussian}}
\label{appdx:proof}

Before getting into the proof, we recall the notion of the the Fourier
transform of a measure. Let $\mu$ be a measure on $(\R^N, \B(\R^N))$. The
Fourier transform of $\mu$ is  given by
\[
   \widehat{\mu}(\hh) = \int_{\R^N} \exp\left(i\ip{\hh}{\xx}\right) \, \mu(d\xx),
\quad \hh \in \R^N.
\]
The Fourier transform uniquely characterizes a 
measure; see e.g.,~\cite[Proposition 1.7]{DaPrato}. It is known that a Gaussian 
measure $\mu = \GM{\cc}{\QQ}$ on $(\R^N, \B(\R^N))$ has Fourier transform
\[
   \widehat{\mu}(\hh) = \exp\big(i \ip{\hh}{\cc}-\frac12 \ip{\QQ\hh}{\hh}\big), \quad 
   \hh \in \R^N.
\]

We are now ready to prove Proposition~\ref{prp:change_meas_gaussian}.
\begin{proof}[Proof of Proposition~\ref{prp:change_meas_gaussian}]
We prove the result by showing that
\[
\widehat{\Law{\vec{Z}}}(\hh) = \exp\big(\ipn{\hh}{\AA \cc + \ss}{M} + 
                      \frac12\ipn{\AA \QQ \AA^\tran \hh}{\hh}{M}\big), 
\quad \text{for all } \hh \in \R^M.
\]
Note that $\ipn{\cdot}{\cdot}{M}$ denotes the Euclidean 
inner product on $\R^M$. 
Let $\hh \in \R^M$. We have
\[
\begin{aligned}
\widehat{\Law{\vec{Z}}}(\hh) &= \int_{\R^M} \exp\big(i\ipn{\hh}{\zz}{M}\big) 
                                \, \Law{\vec{Z}}(d\zz)\\
&= \int_{\R^N} \exp\big(i\ipn{\hh}{\vec{Z}(\xx)}{M}\big) \, \mu(d\xx)\\
&= \int_{\R^N} \exp\big(i\ipn{\hh}{ \AA \xx + \ss}{M}\big) \, \mu(d\xx)\\
&= \exp\big(i\ipn{\hh}{\ss}{M}\big) 
   \int_{\R^N} \exp\big(i\ipn{\hh}{\AA \xx}{M}\big) \, \mu(d\xx)\\
&= \exp\big(i\ipn{\hh}{\ss}{M}\big) \int_{\R^N} 
    \exp\big(i\ipn{\AA^\tran\hh}{\xx}{N}\big) \, \mu(d\xx)\\
&= \exp\big(i\ipn{\hh}{\ss}{M}\big) 
    \exp\big(i\ipn{\AA^\tran\hh}{\cc}{N}-\frac12\ipn{\QQ\AA^\tran\hh}{\AA^\tran\hh}{N}
        \big) \\
&= \exp\big(i\ipn{\hh}{\ss}{M}\big) 
    \exp\big(i\ipn{\hh}{\AA\cc}{M}-\frac12\ipn{\AA\QQ\AA^\tran\hh}{\hh}{M}
        \big) \\
&= 
    \exp\big(i\ipn{\hh}{\AA\cc+\ss}{M}-\frac12\ipn{\AA\QQ\AA^\tran\hh}{\hh}{M}
    \big). \qedhere \\
\end{aligned}
\]
\end{proof}

\end{document}